\documentclass[11pt]{article}
\usepackage{amsmath,amssymb,amsthm,amscd,esint}
\usepackage{url,endnotes,hyperref}
\usepackage{tikz-cd}

\numberwithin{equation}{section}

\newtheorem{theorem}{Theorem}[section]
\newtheorem{lemma}[theorem]{Lemma}
\newtheorem{definition}[theorem]{Definition}

\theoremstyle{corollary}
\newtheorem{corollary}[theorem]{Corollary}
\theoremstyle{conjecture}
\newtheorem{conjecture}{Conjecture}

\theoremstyle{assumption}

\theoremstyle{proposition}
\newtheorem{proposition}[theorem]{Proposition}
\theoremstyle{remark}

\numberwithin{equation}{section}
\everymath{\displaystyle}

\newcommand{\Ric}{\operatorname{Ric}}
\newcommand{\reg}{\operatorname{reg}}

\newcommand{\PP}{\mathbb P}
\newcommand{\CC}{\mathbb C}

\newcommand{\Area}{\operatorname{Area}}
\newcommand{\Rm}{\operatorname{Rm}}
\newcommand{\diam}{\operatorname{diam}}

\begin{document}

\title{Finite Time Singularities of Collapsing K\"ahler Ricci Flow on Ruled Surfaces}
\author{}
\date{}
\maketitle

\vspace{-3em}
\begin{center}
\Large
\begin{tabular}{cc}
Tongxin Xu$^*$ & \qquad Zhenlei Zhang$^\dagger$
\end{tabular}
\end{center}
\vspace{0.1em}

\begin{abstract}
We prove that any finite time collapsing Kähler Ricci flow on ruled surfaces develops a Type I singularity, such singularity is modeled on the standard product shrinker $\mathbb{P}^1\times \mathbb{C}$. As an application, we obtain the optimal collapse rate of fibers on ruled surfaces.
\end{abstract}
\maketitle
\tableofcontents

\section{Introduction}

Ricci flow was introduced by Hamilton  as an evolution equation for Riemannian metrics in 1980s \cite{hamilton1982three}. In this paper, we mainly consider the Kähler Ricci flow. Let $(M,\omega_{0})$ be a compact Kähler manifold of complex
dimension $n$, and let $\omega(t)$ be the solution of the unnormalized
Kähler--Ricci flow
\begin{equation*}
    \frac{\partial}{\partial t}\omega(t)
    =
    -\Ric(\omega(t)),
    \qquad
    \omega(0)=\omega_{0},
\end{equation*}
on its maximal time interval $[0,T)$. Throughout this paper, we assume
that the maximal existence time is finite, namely $T<\infty$. The flow is
said to develop a Type~I singularity at time $T$ if
\begin{equation}\label{eq:type-I}
    \sup_{M}|\Rm(\omega(t))|_{\omega(t)}
    \leq \frac{C}{T-t}
\end{equation}
for some uniform constant $C<\infty$ and every $t\in[0,T)$.

A fundamental problem in the study of finite-time singularities is to
understand the geometry arising under parabolic rescaling. By the work
of Naber and Enders--Müller--Topping, suitable blow-ups based at a
Type~I singular point subconverge to a complete nonflat Ricci shrinker with bounded curvature \cite{Naber,EMT}. In the
Kähler setting, the limit is naturally a Kähler Ricci shrinker.

Set $\alpha_T=[\omega_0]-Tc_1(M)$. A finite-time singularity is called volume-collapsing if
$\alpha_T^n=0$, and the flow becomes extinct at time $T$ if
$\alpha_T=0$. Such singularities are closely related to Fano
fibrations. Song characterized finite-time extinction in the
polarized setting \cite{Song2014FiniteTimeExtinction}, and
Tosatti--Zhang proved that a compact Kähler threefold with a
finite-time collapsing Kähler--Ricci flow admits a Fano fibration
\cite{TosattiZhang2018FiniteTimeCollapsing}. For projective bundles,
Song--Székelyhidi--Weinkove studied finite-time fiber collapse and
Gromov--Hausdorff convergence to the base \cite{SSW}. Related results
on Fano bundles and collapsing rates can be found in
\cite{FuZhang2017FanoBundles,Shen2021FanoBundles,
Fong2011CollapsingRate,ZhangZhang2023CollapsingRate}.

Projective bundle provide a natural setting for
finite-time collapse without extinction. Under some symmetry assumptions,
Song--Weinkove described the finite-time behavior of the
Kähler--Ricci flow on Hirzebruch surfaces
\cite{SongWeinkove2011Hirzebruch}. Fong proved that a class of
$U(1)$-invariant flows on projective bundles develops Type~I
singularities with blow-up models of
$\mathbb C^n\times\mathbb P^1$ type
\cite{Fong2014ProjectiveBundles}. Fong--Tran established Type~I singularities and classified their blow-up models for Kähler--Ricci flows under circle-bundle ansatz \cite{FongTran2026,Fong26}. For related results on collapsing Hirzebruch
surfaces, see \cite{LJT}.

Building on Bamler's compactness theory for metric flows
\cite{Bam20a,Bam23,Bam20b}, Jian--Song--Tian obtained Type~I diameter
and scalar curvature estimates for almost every fiber of a
finite-time collapsing Kähler--Ricci flow on a Fano fiber bundle
\cite{JianSongTian2023FiniteTimeSingularities}. For ruled surfaces, they proved that
Type~I rescalings based at Ricci vertices subconverge smoothly to an
ancient Kähler--Ricci flow on
$\mathbb P^{1}\times\mathbb C$. Related estimates and
regularity results can be found in
\cite{HJST}. Independently, Jian-Song proved that any finite-time collapsing Kähler–Ricci flow on \(\mathbb P^1\)-bundle over projective manifold is of Type I and every tangent flow is the shrinking cylinder \(\mathbb C^m\times\mathbb P^1\) \cite{JS26}.

A folklore conjecture asserts that finite-time singularities of the
K\"ahler--Ricci flow on compact K\"ahler surfaces are always of Type~I.

\begin{conjecture}
Finite time singularities of Kähler Ricci flow on compact Kähler surfaces are all type I singularities.
\end{conjecture}

This conjecture is false in higher dimension \cite{LTZ24,MT23}. 
For surfaces, the volume-noncollapsing case was established by Conlon--Hallgren--Ma,
and Cifarelli--Conlon--Deruelle identified
the noncompact blow-up model with the FIK shrinker
\cite{CCD,CHM25}. In the volume collapsing case, if the flow
becomes extinct, then the surface is del Pezzo and the initial K\"ahler
class is a positive multiple of $c_1(M)$. This case is of Type~I, and
its blow-up models are compact K\"ahler--Ricci shrinkers
\cite{Tia90,Koi90,WZ04,ST08,TZ07,BM87,TZ00}. In the case of volume collapse with no extinction, $M$ is birational to a ruled surface \cite{SWL}, Bamler--Cifarelli--Conlon--Deruelle constructed a toric Type~I flow on
$\operatorname{Bl}_p(\PP^1\times\PP^1)$ whose singular-fibre blow-up
model is the BCCD shrinker on
$\operatorname{Bl}_p(\CC\times\PP^1)$ \cite{BCCD}, but it is still open whether all finite‑time singularities are of Type I in volume collapse case.

Let $M_0$ be a ruled surface i.e. the $\mathbb{P}^1$-bundle over a compact Riemann surface $\Sigma$ with holomorphic submersion $p\colon M_0\to\Sigma$.
 We establish the following results:

\begin{theorem}
Let $(M_0,\omega(t))_{t\in[0,T)}$ be a finite time Kähler Ricci flow on ruled surface $M_0$. If $[\omega_0]-Tc_1(M_0)=[p^*\omega_{\Sigma}]$, then $\omega(t)$ develops a finite time Type I singularity at $t=T$.
\end{theorem}
The same argument applies to $\mathbb P^1$-fibrations. Let $p:M\to\Sigma$ be a holomorphic fibration from a compact Kähler surface onto a compact Riemann surface $\Sigma$, with general fiber biholomorphic to $\mathbb P^1$. Let $\Delta\subset\Sigma$ denote the set of points whose fibers are not biholomorphic to $\mathbb P^1$, we confirm Conjecture 1.1 in the case of regular fibers:
\begin{theorem}\label{main}[Main Theorem]
    Let $(M,\omega(t))_{t\in[0,T)}$ be a finite time Kähler Ricci flow, suppose $M$ admits a holomorphic $\mathbb{P}^1$-fibration. If $[\omega_0]-Tc_1(M)=[p^*\omega_{\Sigma}]$, then 
    \begin{align*}
        \sup\nolimits_{p^{-1}(\mathcal{K})}
 |\Rm_{\omega(t)}|_{\omega(t)}
 \leq \frac{C_\mathcal{K}}{(T-t)}
    \end{align*}
    for any compact set $\mathcal{K}\Subset\Sigma\setminus\Delta$.
\end{theorem}
If every fiber of $p$ is biholomorphic to $\mathbb P^1$, then $\Delta=\varnothing$ and Theorem 1.2 reduces to Theorem 1.1. In the subsequent discussion, we will prove Theorem 1.2 directly.

We also establish the following optimal collapse rate:
\begin{corollary}\label{rate}
Under the same assumptions in Theorem \ref{main}, for every compact set $\mathcal{K}\Subset\Sigma\setminus\Delta$, there exist
nonzero positive constants $c_\mathcal{K}, C_\mathcal{K}$ such that
\[
 c_\mathcal{K}\sqrt{T-t}
 \leq
 \operatorname{diam}
 \bigl(F_q,g(t)|_{F_q}\bigr)
 \leq
 C_\mathcal{K}\sqrt{T-t}
\]
for all $t\in[0,T)$ and $q\in \mathcal{K}$. In particular, if $M=\PP(E)\to\Sigma$ is a ruled
surface, then the estimate holds uniformly
for all $q\in\Sigma$.
\end{corollary}
Let $p\colon M=\mathbb P(E)\to B$ be the projective bundle associated to a holomorphic vector bundle $E\to B$ over a smooth projective variety $B$.
 Under the Type~I curvature bound, the same
argument yields the following.
\begin{theorem}
Let $(M,\omega(t))_{t\in[0,T)}$ be a finite-time Type~I
K\"ahler--Ricci flow on a projective bundle. Suppose that $[\omega_0]-Tc_1(M)=[p^*\omega_B]$. Then, for every sequence $\tau_i\searrow 0$ and any $x\in M$, the
rescaled K\"ahler--Ricci flows $(M,g_i(t),x)$, defined by $g_i(t):=\tau_i^{-1} g\left(T+\tau_i t\right)$ on $t\in[-\tau_i^{-1}T,0)$ subconverge in the pointed $C^\infty$ Cheeger--Gromov topology to the
standard product shrinker $\left(
\mathbb{P}^s\times\mathbb{C}^m,\,
g_{\mathbb{P}^s}\times g_E
\right)$.
\end{theorem}

\paragraph{Outline of the proof of Main Theorem}

The proof has two parts. First, we identify all tangent flows on the regular part of $M$. By the compactness and regularity results of \cite{Bam20b,CHM25}, any tangent flow is modeled by an orbifold  K\"ahler--Ricci shrinker
$(X,g,J,f)$ with isolated singularities. Localizing the estimates of \cite{JianSongTian2023FiniteTimeSingularities} we construct a
global holomorphic function at most linear growth
\[
u:X\longrightarrow\mathbb C,
\qquad
|u(x)|\leq C(1+d_g(x,x_0)),
\qquad
0<c\leq |du|_g\leq C.
\]
The nonvanishing of $du$ makes every orbifold singularity removable.
The splitting lemma with the classification of 1-dim Kähler Ricci shirnker then gives $X=\mathbb P^1\times\mathbb C$. Hence every 
tangent flow is the standard shrinking cylinder.

The second part is to prove the Type I curvature control. We first prove a moving-basepoint entropy lemma by Bamler's 
estimate of pointed Nash entropy, which help us to obtain a limit with constant entropy. If the flow is Type~II, we could choose $x_i\in M$ and
$t_i\nearrow T$ such that $(T-t_i)|\operatorname{Rm}_{g(t_i)}|(x_i)\longrightarrow\infty$. Bamler's compactness theorem and the moving-basepoint entropy lemma give the cylinder $\mathbb P^1\times\mathbb C$ as the
limit. Combining the distance estimate in \cite{CHM25}, we can embedding the limit curve $\mathbb{P}^1\times \{0\}$ in the high curvature region in $M$ to obtain contradiction. Thus the flow is Type~I.
\paragraph{Acknowledgements}
The authors are grateful to Professors Wangjian Jian and Jian Song for pointing out some shortcomings in an earlier version of this paper. We also thank them for bringing to our attention their recent work on closely related results.

\paragraph{Use of AI.} In parts of this work, the authors were assisted by ChatGPT 5.6 Plus. All mathematical statements, hypotheses, references, proofs, and
conclusions were independently verified by the authors, who take full
responsibility for the contents of the paper.

\section{Preliminaries}
\subsection{Ruled surfaces and collapsing fibers}

Let $p:M\to \Sigma$ be a holomorphic fibration from a compact Kähler surface onto a compact Riemannian surface $\Sigma$ with general fiber biholomorphic to $\mathbb P^1$. Let $\Delta\subset \Sigma$ denote the set of points whose fibers are not biholomorphic to $\mathbb P^1$. We consider the finite time Kähler Ricci flow $(M,g(t))_{t\in[0,T)}$ with $[\omega_0]-Tc_1(M)=p^*[\omega_{\Sigma}]$.

The above homology class assumptions strongly restricts the singular fibers. \begin{lemma}\label{lem:singular-fiber}
Every singular fiber of $p$ is the union of two transversely
intersecting $(-1)$-curves.
\end{lemma}

\begin{proof}
We set $F=\textstyle\sum_{\alpha=1}^{N}m_\alpha C_\alpha$, each $C_\alpha$ is a smooth rational curve. We first have
\begin{align*}
0 = F\cdot C_{\alpha} = m_{\alpha}C^2_{\alpha} + \textstyle \sum_{\beta \neq \alpha} m_{\beta} (C_{\alpha}\cdot C_{\beta}),
\end{align*}
which implies $C^2<0$. The adjunction formula gives
\[
0<T^{-1}[\omega_0]\cdot C_\alpha
=c_1(M)\cdot C_\alpha
=2+C_\alpha^2,
\]
so $C_\alpha^2=-1$ and $m_\alpha
=\textstyle\sum_{\beta\ne\alpha}
m_\beta(C_\alpha\cdot C_\beta)$. We set $d_\beta:=\textstyle\sum_{\alpha\ne\beta}C_\alpha\cdot C_\beta$. Summing over $\alpha$ gives $\textstyle\sum_\beta m_\beta(d_\beta-1)=0$, so $d_\beta=1$ for all $\beta$.
This forces $N=2$, $C_1\cdot C_2=1$, and $m_1=m_2=:m$. Then $2=c_1(M)\cdot F=2m$, so $m=1$.
\end{proof}

By the above lemma, we know $M$ is biholomorphic to $\operatorname{Bl}_{p_1,\ldots,p_N}(M_0)$ where $M_0$ is a ruled surface and the points $p_i$ lie on  distinct fibers of $M_0$, we let $\pi\colon \operatorname{Bl}_{p_1,\ldots,p_N}(M_0)\to M_0$ denote the sequence of blow‑ups. For the Kähler–Ricci flow on projective bundles, Song--Székelyhidi--Weinkove establish the two-sided Schwarz type estimate and diameter estimate of fiber \cite{SSW}. In fact, their estimate also holds on the regular part of a $\mathbb{P}^1  $-fibration. 
\begin{lemma}\cite{SSW}
\label{lem:SSW}
Under the above assumptions, suppose that the K\"ahler--Ricci flow $(M,\omega(t))_{t\in[0,T)}$ satisfies $[\omega_0]-T c_1(M)=p^*[\omega_\Sigma]$. Then there exists $c>0$ such that
$c\,p^*\omega_\Sigma \leq\omega(t)$ on $M$. And for any compact set
$\mathcal{K}\Subset \Sigma \setminus\Delta$, there exists $C_\mathcal{K}>0$ such that
\[
 \omega(t)\leq C_\mathcal{K}\omega_0, \quad \sup\nolimits_{q\in K}\operatorname{diam}_{g(t)}F_q
 \leq C_\mathcal{K}(T-t)^{1/3},
 \quad\text{on }p^{-1}(\mathcal{K}) \,\, \text{for}\, \,t\in[0,T).
\]
\end{lemma}

\subsection{Compactness of Ricci flow}
In this section, we review some of Bamler's weak compactness and partial regularity theories that we need \cite{Bam20a,Bam23,Bam20b}. For the Kähler case, we refer to \cite{CHM25}.
\subsubsection{Conjugate heat kernel and pointed Nash entropy}
Let $(M,g(t)_{t\in[0,T)})$ be a Ricci flow on closed manifold. For $(x,t)\in M\times [0,T)$, let $K(x,t,\cdot,\cdot):M\times[0,T)\to (0,\infty)$ denote the conjugate heat kernel based at $(x,t)$. Define probability measures $d\nu_{x,t;s}:=K(x,t;\cdot,s)dg_s$, where $dg_s$ denotes Riemannian volume measure of $(M,g_s)$ for $s\in[0,T)$.

Following \cite[Definition 9.2]{Bam20a}, the $P^*$-parabolic
neighborhood centered at $(x_0,t_0)\in M\times[0,T)$ is defined by
\[
P^*(x_0,t_0;A,-T^-,T^+)
:=
\left\{
(x,t)\in M\times[0,T)
\,\middle|\,
\begin{aligned}
&t\in[t_0-T^-,t_0+T^+],\\
&d_{W_1}^{g_{t_0-T^-}}
\bigl(\nu_{x_0,t_0;t_0-T^-},\nu_{x,t;t_0-T^-}\bigr)<A
\end{aligned}
\right\}.
\]
Here, $d_{W_1}^{g_t}$ denotes the $1$-Wasserstein distance between
probability measures on the metric space $(M,d_{g_t})$; see
\cite[Section 2]{Bam23}.

Fix $x_0\in M$ and let $T_i\nearrow T$. After passing to a
subsequence, $K(x_0,T_i;\cdot,\cdot)$ converges in
$C_{\mathrm{loc}}^\infty(M\times(0,T))$ to a positive solution of the
conjugate heat equation \cite[Lemma 2.2]{MM15},
\[
K(x_0,T;\cdot,\cdot)
:=
\lim_{i\to\infty}K(x_0,T_i;\cdot,\cdot)
\colon M\times[0,T)\longrightarrow\mathbb{R}.
\]
For $t\in[0,T)$, define $d\nu_{x_0,T;t}=K(x_0,T;\cdot,t)\,dg_t$. Then, for any $\varepsilon>0$,
\begin{align*}
    \lim_{i\to\infty}
\sup_{t\in[0,T-\varepsilon]}
d_{W_1}^{g_t}\bigl(\nu_{x_0,T_i;t},\nu_{x_0,T;t}\bigr)=0
\end{align*}
by \cite[Lemma 2.35]{Bam20b}. This $W_1$-convergence also yields
\[
\int_M\int_M d_t^2(x,y)\,
d\nu_{x_0,T;t}(x)\,d\nu_{x_0,T;t}(y)
\leq H_{n}(T-t),
\]
where $H_n:=\frac{(n-1)\pi^2}{2}+4$. Consequently, for each $t\in[0,T)$, there is a point $z\in M$
satisfying $\int_M d_t^2(z,y)\,d\nu_{x_0,T;t}(y)
\leq H_{n}(T-t)$; see \cite[Section 3]{Bam20a}. Such a point $(z,t)$ is called an
$H_{n}$-center of $(x_0,T)$.

We now recall the notion of pointed Nash entropy and some related results.
\begin{definition}\cite[Section 2.6]{Bam20b}
    $K(x_0,T;\cdot,\cdot)$ and $\nu_{x_0,T;t}$ are called conjugate heat kernels based at $(x_0,T)$. Define $f \in C^\infty(M \times [0,t))$ by $K(x_0,T;\cdot,t) = (2\pi\tau)^{-n/2} e^{-f} dg_t$, where $\tau := T-t$. Then the pointed Nash entropy at $(x_0,T)$ is given by
\[
\mathcal N_{x_0,T}(\tau) := \int_M f \, d\nu_{x_0,T;t} - \frac{n}{2}.
\]
\end{definition}

\begin{proposition}\cite[Corollary 5.11]{Bam20a}\label{bamen}
    If $R(\cdot,s) \ge R_{\min}$ and $s < t^* \le t_1,t_2$, $s,t^*,t_1,t_2\in I$, then for $x_1,x_2\in M$
    \begin{align*}
        \mathcal N_{x_1,t_1}(t_1-s)-\mathcal N_{x_2,t_2}(t_2-s)
\le \left(\frac{n}{2(t^*-s)}-R_{\min}\right)^{1/2}
d_{W_1}^{g_{t^*}}\big(\nu_{x_1,t_1}(t^*),\nu_{x_2,t_2}(t^*)\big)
+\frac{n}{2}\log\left(\frac{t_2-s}{t^*-s}\right).
    \end{align*}
\end{proposition}
By the above proposition, we know $\mathcal N_{x_0,T}(\tau)$ is independent of the sequence $T_i\nearrow T$, so $\mathcal N_{x_0,T}(0) := \lim_{\tau \searrow 0} \mathcal N_{x_0,T}(\tau) \in (-\infty,0]$ depends only on $x_0\in M$.
\subsubsection{Compactness and partial regularity of Ricci flow}
We first introduce the notion of orbifold Kähler Ricci shrinker, which appears as a limit in Bamler's compactness theory for Kähler Ricci flow.

\begin{definition}\cite[Definition 2.5]{CHM25}
    An orbifold Kähler Ricci shrinker $(X,g,J,f)$ is a Kähler orbifold $(X,g)$ with complex structure $J$, together with a smooth real--valued function $f\in C^\infty(X)$ satisfying
\[
\operatorname{Ric}_g + \nabla^2 f = g,\qquad \mathcal{L}_{\nabla f}J = 0.
\]
We say $X$ has \textit{isolated singularities} if $X\setminus X_{\mathrm{reg}}$ is discrete. The isotropy group of an isolated singularity is a non--trivial finite subgroup of $U(n)$ acting freely on $\mathbb{C}^n\setminus\{0\}$.
\end{definition}

Because $\nabla f$ vanishes at each point of $X\setminus X_{reg}$, we can define a 1-parameter family of biholomorphisms of $(X,J)$ that preserves $X_{reg}$ by defining $\varphi_{-1}=id_X$ and $\partial_t \varphi_t(x)=\nabla f(\varphi_t(x))/|t|$. Set $g_t:=|t|\varphi_t^*g$, $(X,(g_t)_{t\in(-\infty,0)},J)$ defines an orbifold Kähler Ricci flow. At each orbifold singular point, $X$ is locally biholomorphic to $\mathbb{C}^n/\Gamma$ for some finite subgroup $\Gamma\subseteq U(n)$. When all $\Gamma$ are trivial, we call $X$ a smooth Kähler‑Ricci shrinker.

Let $(M_i,(g_{i,t})_{t\in(-T_i,0]},(x_i,0))$ be a sequence of pointed Ricci flows on compact manifolds of the same dimension $n$ and $T_\infty:=\lim_{i\to\infty} T_i\in(0,\infty]$. By the results of \cite{Bam23} we may pass to a subsequence and obtain $\mathbb F$-convergence on compact time-intervals
\[
\big(M_i,(g_{i,t})_{t\in(-T_i,0]},(\nu_{x_i,0;t})_{t\in(-T_i,0]}\big)
\xrightarrow[\,i\to\infty\,]{ \mathbb F,\mathfrak C }
\big(\mathcal X,(\nu_{x_\infty;t})_{t\in(-T_\infty,0]}\big),
\]
within some correspondence $\mathfrak C$ (for more details see \cite{Bam23}). For some uniform $0<\tau_0<T_\infty$, $Y_0<\infty$ and any $i$, under the following non-collapsing assumption 
\begin{align}\label{noncollapse}
    \mathcal N_{x_i,0}(\tau_0)\ge -Y_0,
\end{align}
 the limit metric flow pair admits a regular-singular decomposition\cite[Theorem 2.4,2.5]{Bam20b}. If it also satisfy 
 \begin{align}
    \lim_{i\to\infty}\mathcal N_{x_i,0}(\tau)=W, \quad \text{for any} \quad \tau\in(0,T)
 \end{align}
 for some constant $W$, then the limit $\big(\mathcal X,(\nu_{x_\infty;t})_{t\in(-T_\infty,0]}\big)$ is a metric soliton \cite[Theorem 2.18]{Bam20b}. In compact 2-dim Kähler Ricci flow, such metric soliton is an orbifold Kähler Ricci shrinker $(X,g,J,f)$. More precisely, there exists an precompact exhaustion $(V_i)$ of $X_{reg}$ along with open embeddings $\psi_i: V_i \to M_i$ such that 
 \begin{align*}
    \psi_i^*g_{i,t}\to g_t, \qquad \psi_i^*J_M \to J, \qquad \psi_i^* \nu_{x_i,0;t}\to \nu_{x_\infty;t}.
 \end{align*}
in $C^{\infty}_{loc}(X_{reg}\times (-\infty,0))$ as $i\to \infty$. See also \cite[Theorem 2.37]{Bam20b}, \cite[Theorem 2.5]{HJ23}, \cite[Theorem 2.8]{CHM25} for more discussion.

We present the following splitting lemma, which will be used in the sequel.

\begin{lemma}\label{splitting}
Let $(X,g,J,f)$ be a smooth complete K\"ahler--Ricci shrinker. Suppose that there is a holomorphic map $U=(u^1,\ldots,u^k)\colon X\longrightarrow\CC^k$ such that 
\begin{align}\label{spl}
    |U(x)|\leq C\bigl(1+d_g(x,x_0)\bigr), \qquad c|\xi|^2
\leq
|\sum\nolimits_{\alpha=1}^k\xi_\alpha\,du^\alpha|_g^2
\leq
C|\xi|^2
\end{align}
for any $\xi\in\CC^k$. Then $X$ is holomorphically isometric to $Y\times \mathbb{C}^k$ where $Y$ is a complete Kähler Ricci shrinker and $\mathbb{C}^k$ is the Gaussian shrinker. 
\end{lemma}

\begin{proof}
We follow the proof of \cite{HeOu}. By \cite[Corollary~1.1]{CZ10}, every polynomial-growth function belongs
to $L^2(X,e^{-f}dV_g)$. We denote this space by $L_f^2$ with inner product $\langle u,v\rangle_f:=\textstyle\int_X u\overline{v}\,e^{-f}\,dV_g$. Set $u^\alpha = v^\alpha + \sqrt{-1}\,w^\alpha$. By shifting each $u^\alpha$ by a constant, we may assume
\[
\int_X u^\alpha e^{-f}dV_g = 0.
\]
Let $g(t)=|t|\varphi_t^*g$, $t<0$, be the canonical Ricci flow associated
with $X$. We first prove
\begin{equation}\label{eq:drift-eigenfunction}
-\Delta_fv^\alpha=v^\alpha,
\qquad
\Delta_f=\Delta-\langle\nabla f,\nabla\cdot\rangle.
\end{equation}
Since $v^\alpha$ is harmonic with respect to every $g(t)$. Hence $\widetilde v^\alpha(x,s)
:=
v^\alpha\bigl(\varphi_{-e^{-2s}}^{-1}(x)\bigr)$ satisfies $\partial_s\widetilde v^\alpha =\Delta_f\widetilde v^\alpha$ for $s\leq0$, combining with (\ref{spl}) we have
\begin{align*}
|\widetilde v^\alpha(x,s)|
&=
\left|
v^\alpha(\varphi_{-e^{-2s}}^{-1}(x))
\right|
\leq
C\left(
1+r\bigl(\varphi_{-e^{-2s}}^{-1}(x)\bigr)
\right)
\leq
Ce^{-s}(1+r(x)).
\end{align*}
The above estimates implies the following $L^2$ estimate
\begin{equation}\label{eq:weighted-growth}
||\widetilde v^\alpha(\cdot,s)||_{L_f^2}^2
\leq Ce^{-2s},
\qquad s\leq0.
\end{equation}

Let ${\varphi_j}$ be an orthonormal eigenbasis of $-\Delta_f$,
with $-\Delta_f\varphi_j=\lambda_j\varphi_j$. The spectrum is discrete and its first positive eigenvalue satisfies
$\lambda_1\geq 1$, see \cite{CZ16,HN14}. Set
\[
a_j^\alpha(s)
:=
\int_X\widetilde v^\alpha(\cdot,s)\varphi_j e^{-f}dV_g.
\]
Self-adjointness of $\Delta_f$ gives
\[
\frac{d}{ds}a_j^\alpha(s)
=
-\lambda_j a_j^\alpha(s),
\qquad
a_j^\alpha(s)=a_j^\alpha(0)e^{-\lambda_js}.
\]
Together with \eqref{eq:weighted-growth}, this implies $|a_j^\alpha(0)|^2e^{-2\lambda_js}\leq Ce^{-2s}$. Letting $s\to-\infty$ shows that $a_j^\alpha(0)=0$ whenever
$\lambda_j>1$. The zero eigenspace is excluded by our
normalization, so we obatin (\ref{eq:drift-eigenfunction}).

By (\ref{eq:weighted-growth}) and (\ref{eq:drift-eigenfunction}), \cite[Proposition 3.2]{XZ26} implies $\nabla^2v^\alpha=0$ and holomorphicity gives $\nabla w^\alpha=J\nabla v^\alpha$, so $\nabla v^\alpha$ and $\nabla w^\alpha$ are parallel. (\ref{spl}) shows that the parallel distribution $\mathcal D
=
\operatorname{span}_{\mathbb R}
\{\nabla v^\alpha,\nabla w^\alpha:1\leq\alpha\leq k\}$ has real dimension $2k$. Since $\mathcal D$ is $J$-invariant, the de Rham decomposition theorem and Kähler condition implies holomorphic isometric splitting $X\cong Y\times\mathbb C^k$. The $\mathbb{C}^k$-factor is flat, and the shrinker equation degenerates to $\sqrt{-1}\partial\bar\partial f_{\mathbb{C}^k}
= \omega_{\mathbb{C}^k}$. This gives the splitting $f(z,w)={|z|^2}/{2}+f_Y(w)$, which implies that $N$ itself is a complete K\"ahler-Ricci shrinker.
\end{proof}
\section{Kähler Ricci flow with collapsing fibers}
\subsection{Tangent flow and shrinking cylinder}

In this section, we consider the Kähler tangent flow of $(M,g(t)_{t\in[0,T)},J_M)$ where $M$ admits a $\mathbb{P}^1$ fibration $p: M\to \Sigma$. Fix a conjugate heat kernel $(\nu_{x_0,T;t})$ based at $(x_0,T)$ with $x_0 \in p^{-1}(\Sigma \setminus \Delta)$. For any sequence $\tau_i>0,\tau_i\searrow0$, we set $g_{i,t}:=\tau_i^{-1}g(T+\tau_it)$, by \cite[Theorem 2.8]{CHM25} we have the following $\mathbb{F}$-convergence
\begin{align*}
    \big(M_i,(g_{i,t})_{t\in[-\tau_i^{-1}T,0)},(\nu_{x_0,T;T+\tau_it})_{t\in[-\tau_i^{-1}T,0)}\big)
\xrightarrow[\,i\to\infty\,]{ \mathbb F,\mathfrak C }
\big(\mathcal X,(\nu_{x_\infty;t})_{t\in(-\infty,0)}\big),
\end{align*}
where $\big(\mathcal X,(\nu_{x_\infty;t})_{t\in(-\infty,0)}\big)$ is an orbifold Kähler Ricci shrinker $(X,g,J,f)$ with isolated sorbifold singularities. We set $\widehat g_i:=g_i(-1)$ in the following arguments. We first construct a global holomorphic function with at most linear growth on $X$ which is the localization of \cite[proof of corollary 2.5,2.6]{JianSongTian2023FiniteTimeSingularities}.

\begin{proposition}\label{prop:base-function}
There exists a global orbifold-holomorphic function $u:X\longrightarrow\CC$ such that
\begin{equation*}
    |u(x)|
    \leq C_1\bigl(1+d_g(x,x_*)\bigr),
    \qquad
    0<c_1\leq |du|_g\leq C_1
    \quad\text{on }X_{\reg}.
\end{equation*}
for some constants $c_1,C_1>0$.
\end{proposition}

\begin{proof}
Fix $x_*\in X_{\reg}$, set $y_i:=\psi_i(x_*)$ and $b_i:=p(y_i)$. After passing to a subsequence, $b_i\to b_\infty\in\Sigma$. Choose
coordinate discs $b_\infty\in U_0\Subset U\Subset\Sigma$ and a holomorphic coordinate $z:U\to\CC$. For any compact set $K\Subset X_{\reg}$, $p(\psi_i(K))\subset U_0$ for all sufficiently large $i$. Joining points of $K$ to
$x_*$ by curves of uniformly bounded $g$-length and using smooth
convergence together with $c\,p^*\omega_\Sigma\leq\omega(T-\tau_i)$ gives
\[
    d_{\omega_\Sigma}
    \bigl(p(\psi_i(x)),b_i\bigr)
    \leq C_K\sqrt{\tau_i},
    \qquad x\in K.
\]
Thus $u_i
    :=
    \tau_i^{-1/2}
    \bigl(z\circ p-z(b_i)\bigr)\circ\psi_i$ are defined on any compact subset of $X_{\reg}$.

By the direct computation:
\begin{equation}\label{eq:rescaled-gradient-identity}
    |du_i|_{\psi_i^*\widehat g_i}^{\,2}
    =
    |d(z\circ p)|_{g(T-\tau_i)}^{\,2}\circ\psi_i .
\end{equation}
The estimates $c\,p^*\omega_\Sigma
    \leq \omega(T-\tau_i)
    \leq C\omega_0$ therefore imply
\[
    C^{-1}|d(z\circ p)|_{g_0}\leq|d(z\circ p)|_{g(T-\tau_i)}
    \leq C|dz|_{\omega_\Sigma}.
\]

Since $p$ is a holomorphic submersion and $dz$ is nowhere zero, $\inf \nolimits_{p^{-1}(\overline{U_0})}
    |d(z\circ p)|_{g_0}>0$, so we have
\begin{equation}\label{eq:uniform-ui-gradient}
    0<c_1
    \leq |du_i|_{\psi_i^*\widehat g_i}
    \leq C_1
\end{equation}
for uniform constants $c_1,C_1>0$ on compact subsets of $X_{\reg}$.

Since $u_i(x_*)=0$, the upper gradient bound gives uniform local
$C^0$ bounds. The functions $u_i$ are holomorphic with respect to
$\psi_i^*J_M$, local elliptic regularity and diagonal
arguments therefore give $u_i\longrightarrow u$ in $C^{\infty}_{loc}(X_{reg})$, where $u$ is
$J$-holomorphic. Passing to the limit in
\eqref{eq:uniform-ui-gradient} yields
\[
    0<c_1\leq |du|_g\leq C_1
    \qquad\text{on }X_{\reg}.
\]
Since the singular set is discrete, the intrinsic length metric on
$X_{\reg}$ has completion $X$, the upper gradient bound and
$u(x_*)=0$ give $|u(x)|\leq C_1d_g(x,x_*)$ for any $x\in X_{\reg}$. Thus $u$ has at most linear growth.

Let $a\in X\setminus X_{\reg}$, and choose an orbifold chart $(\widetilde U,0)/\Gamma\longrightarrow(U_a,a)$. The lift $\widetilde u$ is holomorphic on
$\widetilde U\setminus\{0\}$. Since
$\dim_{\mathbb C}\widetilde U=2$, Hartogs' theorem extends
$\widetilde u$ uniquely across $0$. The extension is
$\Gamma$-invariant by uniqueness and therefore descends to $U_a$.
By continuity, we have  $|d\widetilde u(0)|_{\widetilde g}\geq c_1$. Thus $u$ extends to a global orbifold-holomorphic function on $X$
with nonvanishing differential in every uniformizing chart.
\end{proof}

Then we can prove that the orbifold singularities are removable and identify this shrinker.
\begin{theorem}\label{tangent}
    The tangent flow $X$ is $\mathbb{P}^1\times \mathbb{C}$ with standard product shrinker structure.
\end{theorem}
\begin{proof}
    By Proposition~\ref{prop:base-function},
$d\widetilde u(0)\neq0$ and is $\Gamma$-invariant, its dual
with respect to the $\Gamma$-invariant lifted metric is a nonzero
vector fixed by $\Gamma$.  Since
$\Gamma\subset U(2)$ acts freely on the punctured local cover, which implies
$\Gamma=\{1\}$.  Thus the $X$ is smooth.

Suppose the tangent shrinker $X$ is flat, it is the Gaussian shrinker $\CC^2$, whose pointed Nash entropy $\mathcal N_{x,T}(0)$ is
zero, by \cite[Theorem~2.37]{Bam20b}, there exists a nonempty open set
$U\subset M$ s.t. $g(t)$ extends smoothly to $t=T$.
Choose a holomorphic disc $D\Subset U\cap F_q$, we have $\Area_{g(T)}(D)>0$, contradicting $\Area_{g(t)}(D)
\leq \Area_{g(t)}(F_q)\longrightarrow0$. So $X$ is nonflat.

    Due to $u$ is a nonconstant holomorphic function of linear growth, then by Lemma \ref{splitting}, $X$ is holomorphically isometric to $N\times \mathbb{C}$, where $N$ is a complete complex 1-dimensional Kähler Ricci shrinker i.e. $\mathbb{P}^1$.
\end{proof}
\subsection{Proof of main theorems}
By Theorem \ref{tangent}, we know every tangent flow based at 
$(x,T)$ for $x\in p^{-1}(\Sigma \setminus \Delta)$ is the standard shrinking cylinder
$\PP^1\times\CC$, we denote by $\mathcal N_{\mathrm{cyl}}$ the pointed Nash entropy of the standard shrinking cylinder. We first establish the following moving basepoint entropy lemma.

\begin{lemma}\label{lem:entropy}
For $\tau\in(0,T)$, $\mathcal N_{x,T}(\tau)$ depends only on the selection of $q\in \Sigma \setminus \Delta$ and for every compact set $\mathcal{K}\Subset\Sigma\setminus\Delta$, we have
\begin{equation*}\label{eqen}
 \lim_{\tau\searrow0}
 \sup_{q\in \mathcal{K}}
 \left|
   \mathcal N_{x,T}(\tau)-\mathcal N_{\mathrm{cyl}}
 \right|
 =0, \quad \text{for any}\quad x\in F_q.
\end{equation*}
\end{lemma}

\begin{proof}
Let $t_j\nearrow T$. Fix $\tau\in(0,T)$ and set $s=T-\tau$. By Proposition
\ref{bamen}, we select $t_1=t_2=t^*=t_j$, then we have
\begin{equation}\label{entropy}
 \left|
 \mathcal N_{x,t_j}(t_j-s)
 -
 \mathcal N_{x',t_j}(t_j-s)
 \right|
 \le
 \left(
   \frac{4}{2(t_j-s)}-R_{\min}
 \right)^{1/2}
 d_{g(t_j)}(x,x').
\end{equation}
Since $t_j-s\to\tau$, the coefficient in
\eqref{entropy} is uniformly bounded. If
$x,x'\in F_q$, by Lemma \ref{lem:SSW}, we have
\[
 d_{g(t_j)}(x,x')
 \le \diam_{g(t_j)}F_q
 \longrightarrow0.
\]
Passing to the limit in \eqref{entropy}, we obtain $\mathcal N_{x,T}(\tau)=\mathcal N_{x',T}(\tau)$. Thus $\mathcal N_{x,T}(\tau)$ depends only on the selection of $q\in \Sigma \setminus \Delta$.

Let $U\Subset\Sigma \setminus \Delta$ be a coordinate disc and
choose a smooth section $\sigma\colon U\to M$. After shrinking $U$ if necessary, Lemma \ref{lem:SSW} implies
\[
 d_{\omega_{(t_j)}}\bigl(\sigma(q),\sigma(q')\bigr)
 \le C_U d_{\omega_\Sigma}(q,q'),
 \qquad q,q'\in U.
\]
Applying \eqref{entropy} with
$x=\sigma(q)$ and $x'=\sigma(q')$, and then passing to the limit, gives
\[
 \left|
 \mathcal N_{x,T}(\tau)-\mathcal N_{x',T}(\tau)
 \right|
 \le C_{\tau,U}d_{\omega_\Sigma}(q,q').
\]
Thus $\mathcal N_{x,T}(\tau)$ is continuous on $\Sigma \setminus \Delta$ (it is decied by $q\in \Sigma \setminus \Delta$). Due to the tangent flow on $(x,T)$ is $\mathbb{P}^1\times \mathbb{C}$ for any $x\in M$, we have
\begin{align*}
   \lim_{\tau\searrow0}\mathcal N_{x,T}(\tau)=\mathcal N_{\mathrm{cyl}}
 \qquad\text{for every }q\in \Sigma \setminus \Delta. 
\end{align*}
If $\tau_k\searrow0$, by the monotonicity of the pointed Nash entropy, we know $\mathcal N_{x,T}(\tau_k)$ increase pointwise on 
$\Sigma$ to constant $\mathcal N_{\mathrm{cyl}}$.
Dini's theorem therefore gives uniform convergence along
$\{\tau_k\}$. Since the sequence $\{\tau_k\}$ is arbitrary,
\eqref{eqen} follows.
\end{proof}

\begin{proof}[Proof of Theorem 1.2]
Suppose otherwise. Then there exist $x_i\in p^{-1}(\mathcal{K})$ and $t_i\nearrow T$
such that $(T-t_i)|\Rm_{\omega(t_i)}|_{\omega(t_i)}(x_i)
 \longrightarrow\infty$. Set $\tau_i:=T-t_i$ and $q_i:=p(x_i)$. After passing to a subsequence, we assume $q_i\to q_\infty$. Consider the rescaled flow $\omega_i(t):=\tau_i^{-1}\omega(T+\tau_i t)$ on $t\in[-T/\tau_i,0)$. Let $\bigl(\nu^i_{x_i,0;t}\bigr)_{t<0}$ be the conjugate heat flow for $\omega_i(t)$ obtained by rescaling
a conjugate heat kernel of the original flow based at
$(x_i,T)$. By \cite[Theorem 7.4]{Bam23}, after passing to a subsequence, we have
\[
\left(
  (M,g_i(t))_{t\in[-T/\tau_i,0)},
  (\nu^i_{x_i,0;t})_{t\in[-T/\tau_i,0)}
\right)
\xrightarrow[\displaystyle i\to\infty]{\displaystyle \mathbb F, \mathfrak{C} }
\left(
  \mathcal X,
  (\nu_t)_{t\in(-\infty,0)}
\right).
\]
For every fixed $\tau>0$, the scaling invariance of the pointed Nash
entropy gives $\mathcal N^i_{x_i,0}(\tau)
 =
 \mathcal N_{x_i,T}(\tau_i\tau)$. Since $x_i\in F_{q_i}$, Lemma~\ref{lem:entropy} yields
\[
 \mathcal N^i_{x_i,0}(\tau)
 =
 \mathcal N_{x_i,T}(\tau_i\tau)
 \longrightarrow
 \mathcal N_{\mathrm{cyl}}
 \qquad\text{for any }\tau>0,
\]
as $i \to \infty$. Thus the pointed Nash entropies of the limit are constant for $\tau \in(0,+\infty)$. By \cite[Theorem 2.18, 2.46]{Bam20b} and \cite[Theorem 2.5]{HJ23}, we know the limit $(\mathcal X, (\nu_t)_{t\in(-\infty,0)})$ is an orbifold Kähler Ricci shrinker $(X,g,J,f)$ with isolated orbifold singularities. 

Let $(z_i, t_i)$ be an $H_4$-center of $(x_i,T)$.  By
\cite[Lemma~2.10, Claim 2.12]{CHM25}, We have $d_{\omega_\Sigma}\bigl(p(z_i),q_i\bigr) \leq C\sqrt{\tau_i}$, and for any compact $K\Subset X_{\reg}$, the convergence maps satisfy $d_{\omega_i(-1)}\bigl(\psi_i(x),z_i\bigr)\leq C_K$ for all $x\in K$.
 Since $\omega(t_i)\geq c\,p^*\omega_\Sigma$, it follows that
\begin{align*}
d_{\omega_\Sigma}\bigl(p(\psi_i(x)),q_i\bigr)
&\le
d_{\omega_\Sigma}\bigl(p(\psi_i(x)),p(z_i)\bigr)
+d_{\omega_\Sigma}\bigl(p(z_i),q_i\bigr)\\
&\le
C\,d_{\omega(t_i)}\bigl(\psi_i(x),z_i\bigr)
+C\sqrt{\tau_i}\\
&=
C\sqrt{\tau_i}\,
d_{\omega_i(-1)}\bigl(\psi_i(x),z_i\bigr)
+C\sqrt{\tau_i}\\
&\le C_K\sqrt{\tau_i}.
\end{align*}
for any $x\in K$. Thus, if $z$ is a coordinate near $q_\infty$, then $u_i=\tau_i^{-1/2} \bigl(z\circ p-z(q_i)\bigr)\circ\psi_i$ is well defined on every fixed compact subset of $X_{\reg}$ for sufficiently large $i$ and $u_i$ converge to a global holomorphic function $u$ on $X_{\reg}$. Then the same arguments on Proposition
\ref{prop:base-function} and Theorem \ref{tangent} implies $X$ is smooth and $X$ is holomorphically isometric $\mathbb{P}^1\times \mathbb{C}$ with standard product shrinker structure. 

After an affine transformation, we assume $u(y,w)=w$ and set $Z\cong \mathbb{P}^1=u^{-1}(0)$. Choose a relatively compact tubular neighborhood $\mathcal U$ of
$Z$. Since $u_i\to u$ smoothly on $\overline{\mathcal U}$ for
sufficiently large $i$, the implicit function theorem gives a compact
complex submanifold $Z_i\subset u_i^{-1}(0)\cap\mathcal U$ which is a graph over $Z$. By the definition of $u_i$, we know $\psi_i(Z_i)\subset F_{q_i}$.
The image $\psi_i(Z_i)$ is a nonempty compact complex submanifold of $F_{q_i}\cong\PP^1$ of the same complex
dimension. Hence $\psi_i(Z_i)=F_{q_i}$.

Since $x_i\in F_{q_i}$, there exists $w_i\in Z_i\subset\mathcal U$
such that $\psi_i(w_i)=x_i$. Smooth convergence on the fixed compact
set $\overline{\mathcal U}$ gives
\[
\begin{aligned}
    \tau_i
    \left|\Rm_{\omega(t_i)}\right|_{\omega(t_i)}(x_i)
    =
    |
        \Rm_{\psi_i^*\omega_i(-1)}
    |_{\psi_i^*\omega_i(-1)}(w_i) 
    \leq C_{\mathcal U},
\end{aligned}
\]
contradicting the choice of $(x_i,t_i)$. Therefore the flow is Type~I.
\end{proof}
Combining the above Type I curvature estimates, we derive the following optimal collapse rate of regular fibers.
\begin{proof}[Proof of corollary 1.3]
It is enough to consider $t$ close to $T$. Set $\tau=T-t$ and $\widehat g_t=\tau^{-1}g(t)$, and fix $\mathcal K
\Subset\Sigma\setminus\Delta$. The
Type~I estimate and noncollapsing  give
uniformly bounded geometry around $p^{-1}(\mathcal{K})$. Cover $\mathcal K$ by finitely many coordinate neighborhood $U_a$ and define a local holomorphic function $u_{q,t}=\tau^{-1/2}\bigl(z\circ p-z(q)\bigr)$ on some $p^{-1}(U_a)$ s.t. $q\in U_a$. The localized Schwarz estimates and the scaling of $\widehat g_t$
give
\[
0<c_{\mathcal K}
\leq
|du_{q,t}|_{\widehat g_t}
\leq
C_{\mathcal K}
\]
near $F_q$. Since $u_{q,t}$ is holomorphic and vanishes on $F_q$,
interior elliptic estimates give
\[
|\widehat\nabla^2u_{q,t}|_{\widehat g_t}
\leq C_{\mathcal K}
\qquad\text{on }F_q.
\]

Let $A_{q,t}$ denote the second fundamental form of
$(F_q,\widehat g_t|_{F_q})\subset(M,\widehat g_t)$. Since $F_q=u_{q,t}^{-1}(0)$, the direct computation gives
\[
du_{q,t}\bigl(A_{q,t}(V,W)\bigr)
=
-\widehat\nabla^2u_{q,t}(V,W),
\qquad V,W\in TF_q.
\]
which implies $|A_{q,t}|_{\widehat g_t}\leq C_{\mathcal K}$. Bounded ambient geometry and the
above estimate imply that there exist
$r_{\mathcal K},v_{\mathcal K}>0$ such that
\[
\operatorname{Area}_{\widehat g_t|_{F_q}}
B_{\widehat g_t|_{F_q}}(x,r_{\mathcal K})
\geq v_{\mathcal K},
\qquad
x\in F_q.
\]
On the other hand, the cohomological assumption gives
\[
\operatorname{Area}_{\widehat g_t|_{F_q}}(F_q)
=
\tau^{-1}\int_{F_q}\omega(t)
=
c_1(M)\cdot F_q
=
2.
\]
Packing disjoint balls of radius $r_{\mathcal K}$ along a minimizing
geodesic in $F_q$ therefore gives
\[
\operatorname{diam}(F_q,\widehat g_t|_{F_q})\leq C_{\mathcal K}.
\]
Fix $x\in F_q$ and set $D_{q,t}
=
\operatorname{diam}
\bigl(F_q,\widehat g_t|_{F_q}\bigr)$. Since $F_q=B_{\widehat g_t|_{F_q}}(x,D_{q,t})$ and
$K_{\widehat g_t|_{F_q}}\geq-C_{\mathcal K}$, the Bishop-Gromov volume comparison gives
\[
2
=
\operatorname{Area}_{\widehat g_t|_{F_q}}(F_q)
\leq
\mathrm{Area}_{-C_{\mathcal K}}(D_{q,t}),
\]
where $\operatorname{Area}_{-C_{\mathcal K}}(r)$ denotes the area of a geodesic ball of radius $r$ in the two‑dim space form of constant curvature $-C_{\mathcal{K}}$. Thus
\[
c_{\mathcal K}
\leq
\operatorname{diam}(F_q,\widehat g_t|_{F_q})
\leq
C_{\mathcal K}.
\]
Rescaling gives
\[
c_{\mathcal K}\sqrt{T-t}
\leq
\operatorname{diam}(F_q,g(t)|_{F_q})
\leq
C_{\mathcal K}\sqrt{T-t}.
\]
\end{proof}

Now we discuss the tangent shrinker associated with high dimensional collapsing projective bundles under the Type I curvature assumption. The Type‑I condition naturally yields compactness. More precisely, for every sequence $\tau_i>0, \tau_i\searrow 0$ and any $x\in M$, the pointed rescaled Ricc flows $(M,g_i(t),x)$ defined by $g_i(t):=\tau_i^{-1} g(T+\tau_it)$ on$[-\tau^{-1}_iT,0)$ subconverge in the pointed $C^\infty$-Cheeger--Gromov sense to a complete K\"ahler Ricci shrinker with bounded curvature
\begin{align*}
    \bigl(M, g_j(s), J_M, x\bigr)_{s\in[-\tau^{-1}_iT,0)} 
\xrightarrow{\text{pointed-}C^\infty\text{-Cheeger--Gromov}} 
\bigl(X, g_\infty(s), J_\infty, x_\infty\bigr)_{s\in(-\infty,0)}.
\end{align*}
\begin{lemma}\label{splitting2}
There is a holomorphic map $U=(u^1,\dots,u^m)\colon X\to \mathbb{C}^m$ of at most linear growth such that, for some nonzero positive constants $c_0, C_0$
\begin{equation}
c_0|\xi|^2 \leq |\sum_{\alpha=1}^m \xi_\alpha du^\alpha|_g^2 \leq C_0|\xi|^2
\label{eq:4.1}
\end{equation}
for any $\xi=(\xi_1,\ldots,\xi_m)\in\mathbb C^m$.
\end{lemma}

\begin{proof}
Following the same notation as in Proposition \ref{prop:base-function}, we set $u_i^\alpha = \tau_i^{-1/2}(z^\alpha\circ p - z^\alpha(b_i))\circ \psi_i$ and $U_i:=(u_1^\alpha,\dots,u_m^\alpha)$, where $z^\alpha$ is a holomorphic coordinate on a relatively compact neighborhood of $B$, \eqref{eq:4.1} from the two-sided Schwarz Lemma and smooth convergence. The rest of the proof is the same as in Proposition \ref{prop:base-function}.
\end{proof}

\noindent
\begin{proof}[Proof of Theorem 1.4.]
Combining Lemma \ref{splitting} with Lemma \ref{splitting2}, we obtain a holomorphic isometric splitting $X\cong Y\times \mathbb{C}^m$, where $Y$ is a complete Kähler–Ricci shrinker and $\mathbb{C}^m$ is the  Gaussian shrinker.

Let $F = p^{-1}(b_0)\cong \mathbb{P}^s$ and $\widehat{\omega}_i = \tau_i^{-1}\omega(T-\tau_i)$. We have
\begin{equation}
\mathrm{Vol}_{\widehat{g}_i}(F) = \frac{1}{s!}\int_F \widehat{\omega}_i^s = \frac{(2\pi)^s}{s!}\int_{\mathbb{P}^s} c_1(\mathbb{P}^s)^s =: V_s.
\label{eq:4.2}
\end{equation}
Normalize the splitting so that $X=Y\times \mathbb{C}^m$ and $U(y,z)=z$. By the definition of $U_i$, $\psi_i(U_i^{-1}(0))\subset F$. For every relative compact domain $\Omega$ in $Y$, the implicit function theorem gives graph $\Omega_i\subset U_i^{-1}(0)$ converging smoothly to $\Omega$. Hence
\[
\mathrm{Vol}_{g_Y}(\Omega) = \lim_{i\to\infty} \mathrm{Vol}_{\widehat{g}_i}\big(\psi_i(\Omega_i)\big) \leq \mathrm{Vol}_{\widehat{g}_i}(F) = V_s.
\]
Exhausting $Y$ gives $\mathrm{Vol}(Y,g_Y)\leq V_s$. Since a complete noncompact Ricci shrinker has infinite volume [MW12, Theorem 6.1], $Y$ is compact.

We set $\mathcal{N}_\varepsilon:= Y\times \overline{B_\varepsilon(0)}$. By the implicit function theorem, compact complex submanifold $Y_i:=U_i^{-1}(0)\cap \mathcal{N}_\varepsilon$ is a  graph over $Y$ for all large $i$. By the definition of $U_i$, we have $\psi_i(Y_i)\subset F$. Since $\psi_i(Y_i)$ is compact and has the same complex dimension as $F$, thus $\psi_i(Y_i)=F$ and $Y$ is diffeomorphic to $\mathbb{P}^s$. The Hirzebruch-Kodaira-Yau rigidity theorem \cite{yau1977calabi} and the  Bando-Mabuchi uniqueness theorem \cite{BM87} give $Y\cong \mathbb{P}^s$ biholomorphically with Fubini-Study metric, so $X\cong \mathbb{P}^s \times \mathbb{C}^m$ with the standard product shrinker structure.
\end{proof}

\bibliographystyle{alpha}
\bibliography{ref.bib}
\section*{Author Information}
\noindent\textbf{Tongxin Xu}$^*$\\
School of Mathematical Sciences, Capital Normal University\\
Email: 2250501032@cnu.edu.cn
\vspace{0.5em}

\noindent\textbf{Zhenlei Zhang}$^\dagger$\\
School of Mathematical Sciences, Capital Normal University\\
Email: zhleigo@aliyun.com\\
\end{document}